\documentclass[12pt,english]{article}
\usepackage[T1]{fontenc}
\usepackage{geometry}
\usepackage{amsmath}
\usepackage{amssymb}
\usepackage{amsthm}
\usepackage{color}
\usepackage{graphicx}
\usepackage{epstopdf}
\usepackage{subfigure}
\usepackage{float}
\usepackage{geometry}
\usepackage{authblk}
\usepackage{booktabs}
\usepackage{mathrsfs}
\usepackage{bm}

\usepackage{booktabs}  
\usepackage{array}     
\usepackage{multirow}  

\usepackage[noend]{algpseudocode}
\usepackage{algorithmicx,algorithm}
\makeatletter
\@ifundefined{date}{}{\date{}}

\newcommand{\rom}[1]{\uppercase\expandafter{\romannumeral#1}}

\theoremstyle{plain}\newtheorem{theorem}{Theorem}[section]
\theoremstyle{definition}\theoremstyle{plain}\newtheorem{lemma}[theorem]{Lemma}\newtheorem{corollary}{Corollary}\newtheorem{remark}{Remark}

\usepackage{epstopdf}
\usepackage{subfigure}

\usepackage{physics}

\graphicspath{{fig/}}

\def \e {\mathrm{e}}
\def \i {\mathrm{i}}
\def \d {\mathrm{d}}
\newcommand{\bb}{\boldsymbol}

\makeatother

\usepackage{babel}

\begin{document}
\title{Quantum simulation of Helmholtz equations via Schr{\"o}dingerization}

\author[1,2]{Anjiao Gu}
\author[1,2,3]{Shi Jin}
\author[4,5,6]{Chuwen Ma}

\affil[1]{School of Mathematical Sciences, Shanghai Jiao Tong University, Shanghai 200240, China}

\affil[2]{Institute of Natural Sciences, Shanghai Jiao Tong University, Shanghai 200240, China}
\affil[3]{Ministry of Education Key Laboratory in Scientific and Engineering Computing,  Shanghai 200240, China}
\affil[4]{School of Mathematical Sciences, East China Normal University, Shanghai 200241, China}
\affil[5]{Key Laboratory of MEA, Ministry of Education, East China Normal University, Shanghai 200241, China,
}
\affil[6]{
Shanghai Key Laboratory of PMMP, East China Normal University, Shanghai 200241, China}

\date{}
\maketitle

\abstract{
The Helmholtz equation is a prototypical model for time-harmonic wave propagation. Numerical solutions become increasingly challenging as the wave number $k$ grows, due to the equation’s elliptic yet noncoercive character and the highly oscillatory nature of its solutions, with wavelengths scaling as $1/k$. These features lead to strong indefiniteness and large system sizes.

We present a quantum algorithm for solving such indefinite problems, built upon the Schr\"odingerization framework. This approach reformulates linear differential equations into Schr\"odinger-type systems by capturing the steady state of damped dynamics. A warped phase transformation lifts the original problem to a higher-dimensional formulation, making it compatible with quantum computation.
To suppress numerical pollution, the algorithm incorporates asymptotic dispersion correction. It achieves a query complexity of $\mathcal{O}(\kappa^2\,\text{polylog}\varepsilon^{-1})$, where $
\kappa$ is the condition number and $\varepsilon$  the desired accuracy. 
For the Helmholtz equation, a simple preconditioner further reduces the complexity to $\mathcal{O}(\kappa\, \text{polylog}\,\varepsilon^{-1})$.
Our constructive extension to the quantum setting is broadly applicable to all indefinite problems.
}


\section{Introduction}

Quantum computing has emerged as a promising paradigm to address computational problems that are intractable for classical architectures. As silicon-based processors approach their physical and architectural limits despite decades of progress under Moore’s Law \cite{Freiser1969}, quantum algorithms offer new avenues for accelerating linear algebra and simulation tasks \cite{Feynman1982, DiVincenzo1995quantum}.

A key example is the Harrow-Hassidim-Lloyd (HHL) algorithm, which demonstrates exponential speedup for certain classes of linear systems \cite{Harrow2009Quantum}. Beyond linear solvers, quantum algorithms have shown advantages in integer factorization \cite{Shor1994}, quantum simulation \cite{Deutsch1985quantum, Ekert1998quantum}, and the numerical solution of differential equations \cite{Berry2014Highorder, Montanaro2016FEM, Engel2019Vlasov}. These developments have led to growing interest in quantum numerical methods for ordinary and partial differential equations, which play a central role in physics, engineering, and scientific computing \cite{Costa2019wave, Childs2020Spectral, Jin2022quantum}.

Recent improvements in quantum hardware have further motivated the adaptation of computationally intensive classical algorithms into quantum-compatible frameworks \cite{Andrew1998, Childs2021highprecision}. While significant challenges remain in fault tolerance and scalability, ongoing advances in both hardware and algorithms suggest that quantum computing may eventually provide practical advantages for high-dimensional, ill-conditioned, or oscillatory problems that arise frequently in computational physics.

This work focuses on the numerical solution of the Helmholtz equation, a prototypical example of indefinite partial differential equations. After discretization, the problem reduces to a linear system
 \begin{equation} \label{eq:linear_system}
 A\bm{x} = \bm{b},
 \end{equation}
where $A \in \mathbb{C}^{N \times N}$ ( $N=2^n$ for simplicity) is a sparse indefinite matrix. Classical solvers suffer from high computational costs when $N$ is large (particularly so for high wave number problems), especially due to the ill-conditioning and indefiniteness of $A$. 

Quantum algorithms provide an efficient way to address such high-dimensional systems, requiring only $\mathcal{O}(\log N)$ qubits to encode the system. For special boundary conditions, spectral methods using trigonometric basis functions diagonalize the Laplacian, allowing efficient solution via discrete Fourier, sine, or cosine transforms. These transforms can be implemented on a quantum computer with $\mathcal{O}(\log^2 N)$ operations \cite{Shor1994,Klappenecker2001DST}, improving over the classical $\mathcal{O}(N\log N)$ complexity. However, quantum algorithms for the Helmholtz equation with general boundary conditions remain largely undeveloped.

Recent progress in quantum linear system solvers—such as the HHL algorithm \cite{Harrow2009Quantum} and later improvements using LCU \cite{Childs2017LCU}, QSP \cite{Low2019QSP}, and QSVT \cite{Gilyen2019QSVT}-has led to nearly optimal query complexities for solving general linear systems. Nevertheless, these approaches do not explicitly address the challenges posed by indefinite systems arising from the Helmholtz equations.

Building on recent advances in the Schr\"odingerization framework \cite{Jin2023Quantum,Jin2024PRL}, we extend this approach to indefinite problems and develop a quantum algorithm for the Helmholtz equation. The Schr\"odingerization technique reformulates linear ODEs and PDEs as Schr\"odinger-type equations--with {\it uninary} evolutions-- in a higher-dimensional space, allowing the steady state of an associated damped dynamical system to be simulated via quantum simulation. While previous studies have applied this method to elliptic, parabolic, and hyperbolic problems \cite{Hu2024Multiscale,JLMY25}, this work is, to our knowledge, the first to address indefinite systems.

specifically, we design a quantum algorithm for the Helmholtz equation that incorporates asymptotic dispersion correction and a tailored preconditioner. 
The preconditioner significantly reduces the condition number of the system, yielding an improved query complexity of $\tilde{\mathcal{O}}(\kappa(A)\text{polylog}(\varepsilon^{-1}))$, where $\kappa(A) = \|A\|\|A^{-1}\|$ is the condition number of the system matrix and $\varepsilon$ denotes the target accuracy.
 Here, $\tilde{\mathcal{O}}$ hides polylogarithmic factors. This enables efficient simulation of indefinite wave problems within a quantum computing framework.

The remainder of this paper is organized as follows.
Section \ref{sec:Helmhlotz} introduces the Helmholtz equation, along with quantum spectral methods and classical asymptotic dispersion correction techniques.
Section \ref{sec:Schrforindefinite} presents the Schr\"odingerization framework applied to indefinite systems using damped dynamical reformulations.
In Section \ref{sec:implementation and complexity}, we analyze the computational complexity of the proposed quantum algorithm and compare it with classical methods.
Section \ref{sec:Quantum_preconditioning_method} develops a quantum algorithm for the preconditioned Helmholtz system.
Moreover, in Section \ref{sec:numerical tests}, we display the numerical results for different arguments.
Finally, we conclude this paper in Section 7.

\section{The Helmholtz equations}\label{sec:Helmhlotz}

The Helmholtz equation, 
\begin{equation}\label{eq:helmholtz}
-\Delta{u}-k^2 u=f \quad \text{in} \quad \Omega \subset \mathbb{R}^d,
\end{equation}
where $k$ is the wavenumber and $f$ represents source terms, governs wave propagation phenomena in acoustics, electromagnetics, and seismology. 
Its numerical solution presents unique challenges due to the \textit{indefinite} nature of the operator and \textit{highly oscillatory} solutions when $k \gg 1$. 
In this paper, we consider the first-order Sommerfeld radiation condition
\begin{equation}\label{eq:Sommerfel BD}
    \frac{\partial u}{\partial \bm{n}} - i k u = 0, \quad \text{on} \;\text{part}\; \text{of}\; \partial \Omega,
\end{equation}
where $\bm{n}$ denotes the outward direction normal to the boundary. 
More accurate radiation conditions can be found in, for example, \cite{Reps2010}.
On the remaining portion of the boundary, we impose Dirichlet boundary conditions.

Classical approaches include finite difference/element methods (FDM/FEM) which generate sparse linear systems, but suffer from pollution errors at high wavenumbers. 
Modern developments focus on robust preconditioning techniques, particularly the shifted Laplacian preconditioner $P=-\Delta{u}-(k^2 + i\alpha)$ combined with Krylov subspace methods or multigrid methods. 
To resolve oscillations, the mesh size $h$ must satisfy $kh \ll 1$, leading to \textit{excessively large matrices} for high $k$ after numerical discreitzations. 
This exacerbates the condition number issue and memory footprint.
Standard preconditioners (e.g., ILU, SSOR) fail for large $k$. 
Designing effective Helmholtz-specific preconditioners (e.g., shifted Laplacian, domain decomposition) is non-trivial for quantum computing.
To our knowledge, there are no quantum algorithms for the Helmholtz equation, let alone well-preconditioned quantum preconditioning systems.
The quest for wavenumber-robust solvers continues to drive research in both classical and quantum computational approaches.

\subsection{Asymptotic dispersion correction for Helmholtz equations}\label{discretization}

With regard to more general boundary conditions, we will consider the construction of quantum algorithms in this section.
For continuous Helmholtz problems, plane wave solutions are exactly represented as $\mathrm{e}^{\mathrm{i}kx\cdot\theta}$ (where $\theta\in S^{d-1}$). 
However, under numerical discretization, these waves adopt a modified form $\mathrm{e}^{\mathrm{i}k_dx\cdot\theta}$, where the discrete wavenumber $k_d(\theta,h)$ deviates from $k$ due to mesh-dependent numerical artifacts. 
This phenomenon, termed dispersion error, fundamentally contributes to the \emph{pollution effect} \cite{Ihlenburg1995FEM,Melenk2011Galerkin,Zhu2013FEM,Spence2023FEM}-wherein numerical accuracy degrades with increasing wavenumber despite maintaining $kh$ constant.

In this paper, we adopt the dispersion correction technique for FD discretizations of Helmholtz problems \cite{Cocquet2024Asymptotic}. 
The technique is based on a modified wavenumber, which can be obtained in closed form for arbitrary FD discretizations by obtaining the extrema of an associated function defined on a compact set. 
This function is simply obtained from the Taylor expansion of the discrete symbol of the FD stencil considered.
Specifically, this method is based on the expansion of the discrete wavenumber $k_d$ as the meshsize goes to zero. 
A shifted wavenumber is next introduced in the stencil to minimize the leading-order term in the expansion of $(k_d-k)$. 
The literature shows that the shifted wavenumber can be determined in closed form by computing the extrema of the remainder, which is a trigonometric polynomial.

For one-dimensional Helmholtz equations, we consider a uniform grid and assume that $kh\notin\pi\mathbb{N}$.
Using a $3$-point stencil for the second-order derivative, the discrete problem associated to \eqref{eq:helmholtz} reads
\begin{equation}\label{1dcorrectionFDM}
-\frac{1}{h^2}(u_{i-1}-2u_i+u_{i+1})-\widehat{k}^2 u_i=f(x_i),\ i=1,\cdots,n,
\end{equation}
where $\widehat{k}$ is the shifted wave number given by
$$\widehat{k}=\sqrt{\frac{2}{h^2}(1-\cos(kh))}.$$
Inserting $u_j:=e^{ik_d x_j}$ into \eqref{1dcorrectionFDM} with $f=0$ and neglecting the boundaries, it can be observed that the discrete wave number $k_d$ is a solution to 
$$\cos(k_d h)=1-\frac{\widehat{k}^2h^2}{2}.$$
Thus, this scheme yields $k_d=k$ and does not have any dispersion error.
For matrix $A$ which is derived from \eqref{1dcorrectionFDM}, we use
\begin{equation}\label{matA}
A=\left[\begin{array}{ccc}
    -1 & 2-\widehat{k}^2h^2 & -1
\end{array}\right],
\end{equation} 
which leads to $\Vert A\Vert_{max}=\mathcal{O}(1)$ and $\bb{b}=\left[h^2f(x_i)\right]$ for interior points.
For boundary points, we employ either the backward difference approximation or the ghost point method.

For higher dimensional Helmholtz problems, authors in \cite{Cocquet2024Asymptotic} give the easy-to-use closed form formulas for the asymptotically optimal shift associated to the second-order 5-point scheme and a sixth-order 9-point scheme in two dimensions, and the 7-point scheme in three dimensions that yield substantially less dispersion error than their standard (unshifted) version.
It also shows that for a sufficiently small meshsize, reducing the dispersion error also reduces the relative error in the solution.

\section{Schr{\"o}dingerization for the indefinite systems}\label{sec:Schrforindefinite}

In this section, we address the Schr\"odingerization of the indefinite system \eqref{eq:linear_system}. Inspired by \cite{JLMY25}, we first derive a linear ordinary differential equation (ODE) whose steady-state solution coincides with the exact solution of \eqref{eq:linear_system}, under the conditions that $A$ is nonpositive but invertible. We then apply the Schr\"odingerization technique to this ODE system to formulate a quantum algorithm.

\subsection{Steady-state solution of linear ODEs}
A natural idea for solving an indefinite system is to perform a regularization $B:=A^{\dagger}A$ to make sure that $B$ is positive-definite.
When solving linear systems 
$$B\bm{x}=\bm{f},\ \bm{f}:=A^{\dagger}\bm{b},$$
gradient flow (GF) methods can reformulate the convergent stationary iterative algorithm as a linear ODE system, where the solution to the algebraic problem corresponds to the steady state of the ODE. 
It follows the ODE
\begin{equation}\label{steadystateeq}
\frac{d\tilde{\bm{x}}}{dt}=-B\tilde{\bm{x}}+\bm{f}.
\end{equation}
The solution $\tilde{\bm{x}}(t)$ converges to $\bm{x}$ at a rate of $\mathcal{O}(e^{-\sigma_{\min}^2(A)t})$ according to the following lemma. The proof can be found in \cite{JLMY25}.

\begin{lemma}
Given $B:=A^{\dagger}A$ for some $A\in\mathbb{C}^{N\times N}$, then $\tilde{\bm{x}}(t)$ converges to the steady state $\bm{x}$ in the sense of 
\begin{equation}\label{errforss}
\Vert\tilde{\bm{x}}(t)-\bm{x}\Vert\leq e^{-\sigma_{\min}^2(A)t}\Vert\tilde{\bm{x}}(0)-\bm{x}\Vert.
\end{equation}
\end{lemma}
It is evident that $\tilde{\bm{x}}(t)$ converges regardless of the choice of initial condition $\tilde{\bm{x}}(0)$. For simplicity, we initialize the system with $\tilde{\bm{x}}(0) = \bm{0}$. 

For ill-conditioned matrices (e.g. When $k^2$ approaches an eigenvalue of the discrete Laplacian), GF suffers from slow convergence due to the dominance of small eigenvalues.
Therefore, damped dynamical systems (DDS) will have more advantages in this case.
The DDS introduces inertia via a second-order ODE:
\begin{equation}\label{steadystateeq2}
\frac{d^2\boldsymbol{v}(t)}{dt^2}+\gamma\frac{d\boldsymbol{v}(t)}{dt}=-B\boldsymbol{v}(t)+\boldsymbol{f},
\end{equation}
achieving an accelerated convergence rate of $\mathcal{O}(e^{-\sigma_{\min}(A)t})$ where $\sigma_{\min}(A)$ is the smallest singular value of $A$. 
This is particularly advantageous for ill-conditioned problems, as the dependence on the condition number reduces from $\kappa(B)=\kappa^2(A)$ to $\sqrt{\kappa(B)}=\kappa(A)$ .
\begin{lemma}\label{thm:steady state of second eq}
For system \eqref{steadystateeq2}, Under critical damping ($\gamma=2\sigma_{\min}(A)$), the global convergence rate is optimal:
\begin{equation}\label{errforss2}
\Vert \boldsymbol{v}(t)-\boldsymbol{x}\Vert\leq e^{-\sigma_{\min}(A)t}\Vert\bm{v}(0)-\bm{x}\Vert.
\end{equation}
\end{lemma}

\begin{proof}

Let $e(t)=\boldsymbol{v}(t)-\boldsymbol{x}$, it yields the homogeneous equation
$$\ddot{e}+\gamma\dot{e}+Be=0.$$
Perform eigendecomposition $B=Q\Lambda_B Q^T$, and let $\tilde{e}=Q^Te$, yielding the decoupled equations 
$$\ddot{\tilde{e}}_i+\gamma\dot{\tilde{e}}_i+\Lambda_B \tilde{e}_i=0,\ i=1,\cdots,N,\ \Lambda_B=\text{diag}\{\lambda_1,\cdots,\lambda_N\}.$$
For each mode, the characteristic equation is 
$$r^2+\gamma r+\lambda_i=0,$$ 
with roots 
$$\frac{-\gamma\pm\sqrt{\gamma^2-4\lambda_i}}{2}.$$
Then, we can get three cases:
\begin{itemize}
\item Case 1: Critical Damping ($\gamma=2\sqrt{\lambda_i}$) with convergence rate $\mathcal{O}(e^{-\sqrt{\lambda_i}t})$;
\item Case 2: Over damping ($\gamma>2\sqrt{\lambda_i}$) with convergence rate $\mathcal{O}(e^{-\left(\gamma-\sqrt{\gamma^2-4\lambda_i}\right)t/2})$;
\item Case 3: Under damping ($\gamma<2\sqrt{\lambda_i}$) with convergence rate $\mathcal{O}(e^{-\gamma t/2})$.
\end{itemize}
Since the system's convergence rate is determined by the slowest mode, it yields the overall convergence rate $\mathcal{O}(e^{-\sigma_{\min}(A)t})$ by choosing 
$$\gamma=2\sqrt{\lambda_{\min}(B)}=2\sigma_{\min}(A).$$

\end{proof}

For practical applications, we can utilize the inequality
\[
\sigma_{\text{min}}(A)\leq\vert\lambda_i(A)\vert\leq\sigma_{\text{max}}(A).
\]
When $A$ is a normal matrix, we further have the exact relationship
\[
\sigma_i(A) = |\lambda_i(A)|,
\]
which leads to $\sigma_{\min}(A)=\min\{\vert\lambda_i(A)\vert\}.$
Moreover, second-order equation \eqref{steadystateeq2} can be further rewritten as a equivalent first-order equation, starting from zero,
\begin{equation}\label{1stODE}
\frac{d\bm{V}}{dt}=M \bm{V}+F, \quad \bm{V}(0) = \textbf{0},
\end{equation}
by doubling the dimensionality by defining
\begin{equation*}
\bm{V}=\left[\begin{array}{c}
    \bm{v}\\
    \bm{w}
\end{array}\right],\ 
M=\left[\begin{array}{cc}
    O_N & -A^{\dagger}\\
    A & -\gamma I_N
\end{array}\right],\ 
F=\left[\begin{array}{c}
    \bm{0}\\
    -\bm{b}
\end{array}\right].
\end{equation*}
Equation \eqref{1stODE} can also be considered as an approximation to linear equation:
\begin{equation*}
\left[\begin{array}{cc}
    O_N & -A^{\dagger}\\
    A & -\gamma I_N
\end{array}\right]\left[\begin{array}{c}
    \bm{x}\\
    \bm{0}
\end{array}\right]=\left[\begin{array}{c}
    \bm{0}\\
    \bm{b}
\end{array}\right].
\end{equation*}

Since \eqref{1stODE} and \eqref{steadystateeq2} are equivalent, from Lemma \ref{thm:steady state of second eq}, the relation between the  steady-state solution to \eqref{1stODE} and the exact solution  is described below.
\begin{theorem}\label{thm:ODESS}
	Assume that $A$ is indefinite and invertible. The ODE system \eqref{1stODE} starts from zero. Then for any $\varepsilon>0$, we have
	\begin{equation*}
		\|\bm{v}(T)-\bm{x}\| \leq \varepsilon\|\bm{x}\|,
	\end{equation*}
	where the stopping time $T$ for the evolution satisfies
	\begin{equation}\label{eq:stop T sym}
		T\geq \frac{1}{\sigma_{\min}(A)} \log \frac{1}{\varepsilon}.
	\end{equation}
\end{theorem}

\begin{remark}
 \eqref{steadystateeq2} can also be rewritten as another first-order equation
\begin{equation*}
\boldsymbol{w}=\left[\begin{array}{c}
    \bm{v}\\
    \dot{\bm{v}}
\end{array}\right],\ 
M=\left[\begin{array}{cc}
    O_N & I_N\\
    -B & -\gamma I_N
\end{array}\right],\ 
F=\left[\begin{array}{c}
    \bm{0}\\
    A^{\dagger}\bm{b}
\end{array}\right].
\end{equation*} 
This formulation is more consistent with the physical intuition.
However, due to the conservative part being a non-unitary evolution, Schr{\"o}dingerization will disrupt this physical property, resulting in a significant increase in the cost of recovery and even the loss of quantum advantage.
Numerically, the presence of positive eigenvalues in $H_1$ forces an error-amplifying factor $\exp(\mathcal{O}(k^2))$ during solution recovery, rendering the approach impractical.
\end{remark}

\begin{remark}
On the other hand, it also provides additional route to solve a class of Hamiltonian mechanics with Hamiltonian 
$$H=\frac{1}{2}(p^T p+q^T B q)$$ 
where $B=A^{\dagger} A$ for $A\in\mathbb{C}^{N\times N}$.
Specifically, we can transform the original equation
\begin{equation*}
\left[\begin{array}{c}
    \dot{\bm{q}}\\
    \dot{\bm{p}}
\end{array}\right]=\left[\begin{array}{cc}
    O_N & I_N\\
    -B & O_N
\end{array}\right]\left[\begin{array}{c}
    \bm{q}\\
    \bm{p}
\end{array}\right]
\end{equation*} 
to a new Hamiltonian system with unitary evolution:
\begin{equation*}
\left[\begin{array}{c}
    \dot{\bm{q}}\\
    \dot{\bm{\tilde{p}}}
\end{array}\right]=\left[\begin{array}{cc}
    O_N & -A^T\\
    A & O_N
\end{array}\right]\left[\begin{array}{c}
    \bm{q}\\
    \bm{\tilde{p}}
\end{array}\right].
\end{equation*} 
If one only cares about the position of the motion, this is a method that can be directly applied to quantum computing based on unitary evolution.

\end{remark}

\subsection{Description of Schr{\"o}dingerization}\label{review}

Above reformulations enable the application of Schr{\"o}dingerization to construct a Hamiltonian system for quantum computing.
With the task of solving the linear system \eqref{1stODE}, we introduce a quantum algorithm to solve this problem. 
First, equation \eqref{1stODE} can be further rewritten as a homogeneous system:
\begin{equation}\label{homoeq}
\frac{d\bm{V}_f}{dt}= M_f\bm{V}_f,\quad
\bm{V}_f=\left[\begin{array}{c}
    \bm{V} \\
    \bm{r}
\end{array}\right],\quad
M_f=\left[\begin{array}{cc}
    M & \frac{I}{T} \\
    O & O
\end{array}\right],\quad
\bm{V}_f(0)=\left[\begin{array}{c}
    \bm{0} \\
    TF
\end{array}\right],
\end{equation}
where we set $\bm{V}_0=\bm{0}$ and $T$ is the final time satisfying \eqref{eq:stop T sym}.
Then, 
$M_f$ can be further decomposed into a Hermitian term and an anti-Hermitian one:
\begin{equation}\label{Hmat}
M_f=H_1+iH_2,\quad H_1=\frac{1}{2}\left[\begin{array}{cc}
    2 M_D & \frac{I}{T} \\
    \frac{I}{T} & O
\end{array}\right],\quad H_2=\frac{1}{2i} \left[\begin{array}{cc}
    2 M_H & -\frac{I}{T} \\
    \frac{I}{T} & O
\end{array}\right],
\end{equation} 
where $M_D$ and $M_H$ are defined by 
\begin{equation*}
	M_D = \frac{1}{2}(M+M^{\dagger}) = \begin{bmatrix}
		O &O\\
		O &-\gamma I 
	\end{bmatrix},\quad 
	M_H = \frac{1}{2}(M-M^{\dagger}) = \begin{bmatrix}
		O & -A^{\dagger}\\
		A & O
	\end{bmatrix}.
\end{equation*}

Using the warped phase transformation $\bm{W}(t,p)=e^{-p}\bm{V}_f(t)$ for $p>0$ and properly extending the initial data to $p\le 0$, equation \eqref{homoeq} are then transferred to linear convection equations:
\begin{equation}\label{linearconv}
\frac{d\bm{W}}{dt}=-H_1\partial_p\bm{W}+iH_2\bm{W},\quad \bm{W}(0)=\psi(p)\bm{W}(0).
\end{equation}
Here, the smoothness of $\psi(p)$ has an impact on the convergence rate of the numerical method.
For instance, if we use $\psi(p)=e^{-\vert p \vert}$, it implies a first-order accuracy on the spatial discretization due to the regularity in $p$ of the initial condition.
And if one uses a smoother initial value of $\bm{v}(0)$ with
\begin{equation}\label{xip}
\psi(p)=
\begin{cases}
(-3+3e^{-1})p^3+(-5+4e^{-1})p^2-p+1,& p\in(-1,0),\\
e^{-\vert p\vert},& \text{otherwise},
\end{cases}
\end{equation}
it gives a second-order accuracy~\cite{Jin2024schrodingerization, Jin2024illposed}. The exponential accuracy  can also be achieved by requiring smooth enough $\psi(p)\in C^{\infty}(\mathbb{R})$, which results in optimal complexity. see \cite{JLMY252}.

\subsubsection{Recovery of the solution}\label{retrieval}

If all eigenvalues of $H_1$ are negative, the convection term of \eqref{linearconv} corresponds to a wave moving from the right to the left, thus it does not need to impose a boundary condition for $\bm{W}$ at $p=0$. 
For the general case, i.e. if $H_1$ has non-negative eigenvalues, one uses the following Theorem to recover the original solution.
\begin{theorem}\label{Thforp}
	\cite{Jin2024schrodingerization} Assume the eigenvalues of $H_1$ in \eqref{Hmat} satisfy
	$$\lambda_1(H_1)\le\lambda_2(H_1)\le\cdots\le\lambda_{4N}(H_1),$$ 
	then solution $\bm{V}_f(T)$ can be restored by
	\begin{equation*}
		\bm{V}_f=e^{p}\,\bm{W}(T,p),\ \textit{or}\quad \bm{V}_f=e^{p}\int_{p}^{\infty}\bm{W}(T,q)dq,
	\end{equation*}
	where $p\ge p^{\Diamond}=\max\{\lambda_{4N}(H_1)T,0\}$.
\end{theorem}
For \eqref{Hmat}, it is easy to check that $p^{\Diamond}=\frac{1}{2}$ which leads to a small cost for recovery.

\subsection{The discrete Fourier transform for Schr{\"o}dingerization}

To discretize the $p$ domain, we choose a finite domain $[-L,R]$ with $L,R>0$ large enough satisfying
\begin{equation}\label{eq: L,R,criterion}
	e^{-L + \lambda^-_{\max}(H_1) T} \leq \varepsilon,\quad
	e^{-R + \lambda^+_{\max}(H_1) T} \leq \varepsilon,
\end{equation}
where
\[\lambda_{\max}^-(H_1) :=\sup\{|\lambda_j|:\lambda_j\in \lambda(H_1), \lambda_j<0\},\]
\[\lambda_{\max}^+(H_1) :=\sup\{|\lambda_j|:\lambda_j\in \lambda(H_1), \lambda_j>0\},\]
and $\varepsilon$ is the desired accuracy, $\lambda(H_1)$ is the set of the eigenvalue of $H_1$.
Then the wave initially supported inside the domain remains so in the duration of computation, and we can use spectral methods to obtain a Hamiltonian system for quantum computing.

We adopt the recovery strategy in Theorem \ref{Thforp} to restore the solution for the Schr\"odingerization.
\begin{theorem}\label{thm:recovery u}
Assume that $A$ is indefinite and invertible.
The vector $\bm{v}(T)$ is the solution to the ODE system \eqref{1stODE}.
To control the overall accuracy within $\varepsilon$ such that
\begin{equation*}
\|\bm{v}(T)-\bm{x}\|/\|\bm{x}\| \leq \varepsilon,
\end{equation*}
with $T$ satisfying \eqref{eq:stop T sym}, 
we recover $\bm{v}(T)$  by
\begin{equation}\label{eq:recover by one point}
\bm{v}(T) = e^{p}(\bra{0}^{\otimes 2} \cdot I_N) \bm{W}(T,p) \quad
\mbox{for any $p\geq p^{\Diamond} = \dfrac12$}.
\end{equation}
The truncation of the $p$-domain  satisfies
\begin{equation}\label{LR}
L =\mathcal{O}\big(\kappa(A)\log \frac{1}{\varepsilon}+\frac 1 2\big),\quad  R = \mathcal{O} \big(\log \frac{1}{\varepsilon} +\frac{1}{2}\big).
\end{equation}
\end{theorem}

When one truncates the extended region to a finite interval with $L$ and $R$ satisfying \eqref{LR}, one can apply the periodic boundary condition in the $p$ direction and use the Fourier spectral method by discretising the $p$ domain.
We set the uniform mesh size  $\triangle p = (R+L)/N_p$, where $N_p=2^{n_p}$ is a positive even integer. The grid points are denoted by $-L=p_0<\cdots<p_{N_p}=R$. The one-dimensional basis functions for the Fourier spectral method are usually chosen as
\begin{equation} \label{eq:phi nu}
	\phi_l(p) = e^{i\nu_l  (p+L)},\quad \nu_l  = 2\pi (l-N_p/2)/(R+L),\quad 0\leq l\leq N_p-1.
\end{equation}
Using \eqref{eq:phi nu}, we define
\begin{equation}
	\Phi  = (\phi_{jl} )_{N_p\times N_p} = (\phi_l (p_j))_{N_p\times N_p},  \quad
	D_{p} = \text{diag}(\nu_0 ,\cdots,\nu_{N_p-1} ).
\end{equation}

Define the vector $\bm{W}_h$ as the collection of the approximation values of the function $\bm{v}$ at the grid points, given by
\begin{equation}
	\bm{W}_h =\sum_{k=0}^{N_p-1}\sum_{j=0}^{4N-1} W_{kj}(t) \ket{k}\ket{j},
\end{equation}
where $W_{kj}(t)$ denotes the approximation to $W_j(t,p_k)$, the $j$-th component of $\bm{W}(t,p_k)$.
Considering the Fourier spectral discretization on $p$, one easily gets a "schr\"odingerized" system:
\begin{equation} \label{eq:hamiltonian}
	\begin{aligned}
		\frac{\rm d}{{\rm d} t} \bm{W}_h &= -\i(P\otimes H_1) \bm{W}_h + \i(I_{N_p} \otimes H_2)\bm{W}_h\\
		& = -i (\Phi  \otimes I)
		( D_p\otimes H_1 - I_{N_p} \otimes H_2 )
		(\Phi ^{-1}\otimes I) \bm{W}_h.
	\end{aligned}
\end{equation}
Here $P = \Phi  D_{p} \Phi ^{-1}$ is the matrix representation of the momentum operator $-i\partial_p$.
At this point, a quantum simulation can be carried out on the Hamiltonian system above:
\begin{equation*}
	\ket{\bm{W}_h(T)} = (\Phi \otimes I)~
	\mathcal{U}(T) ~
	(\Phi ^{-1}\otimes I) \ket{\bm{W}_h(0)},
\end{equation*}
where $\mathcal{U}(T)$ is a unitary operator, given by
\begin{equation*}
	\mathcal{U}(T) = \e^{-\i H T}, \qquad H:= D_p\otimes H_1 - I_{N_p}\otimes H_2,
\end{equation*}
and $\Phi $ (or $\Phi^{-1}$) is completed by (inverse) quantum Fourier transform (QFT or IQFT).

\section{Implementation and complexity analysis}\label{sec:implementation and complexity}

In this section, we give the detailed implementation of the Hamiltonian simulation of $ \mathcal{U}(T) = \e^{-i H T} $, where sparse access to the Hamiltonian $ H $ is assumed.
This can be done, since the sparsity of $H$ is almost the same as the sparsity of $A$.

\subsection{Implementation of  Hamiltonian simulation}\label{subsec:implementation}
To circumvent the expense of encoding $ D_p $,
one can express the evolution operator $ \mathcal{U}(T) $ as a select oracle
\begin{equation*}\label{selectVk}
	\mathcal{U}(T) = \sum_{k=0}^{N_p-1}  \ket{k}\bra{k} \otimes e^{-\i (\nu_k H_1- H_2) T}
	=: \sum_{k=0}^{N_p-1} \ket{k}\bra{k} \otimes V_k(T).
\end{equation*}
Since the unitary $V_k(T)$ corresponds to the simulation of the Hamiltonian $H_{\nu_k}:=\nu_k H_1 - H_2$, we assume the block-encoding oracles encoding the real and imaginary parts separately, namely
\[(\bra{0}_a\otimes I)U_{H_i} (\ket{0}_a\otimes I) = \frac{H_i}{\alpha_i}, \qquad i = 1,2, \]
where $\alpha_i\ge \|H_i\|$ is the block-encoding factor for $i = 1,2$.

According to the discussion in \cite[Section 4.2.1]{ACL2023LCH2}, there is an oracle $\text{HAM-T}_{H_{\nu}}$ such that
\begin{equation}\label{HAMT}
	(\bra{0}_{a'} \otimes I) \text{HAM-T}_{H_\nu}(\ket{0}_{a'} \otimes I)
	= \sum_{k=0}^{N_p-1} \ket{k}\bra{k} \otimes \frac{H_{\nu_k}}{\alpha_1 \nu_{\max} + \alpha_2},
\end{equation}
where $H_{\nu_k} = \nu_k H_1 - H_2$ and $\nu_{\max} = \max_k |\nu_k|$ represent the maximum absolute value among the discrete Fourier modes. This oracle only uses $\mathcal{O}(1)$ queries to block-encoding oracles for $H_1$ and $H_2$.
With the block-encoding oracle $\text{HAM-T}_{H_\nu}$, we can implement
\[\text{SEL}_0 = \sum_{k=0}^{N_p-1} \ket{k}\bra{k} \otimes  V_k^a(T), \]
a block-encoding of $\mathcal{U}(T)$, using the quantum singular value transformation (QSVT) \cite{gilyen2019quantum} for example, where $V_k^a(T)$ block-encodes $V_k(T)$ with $\|V_k^a(T) - V_k(T)\| \le \delta$.
This uses the oracles for $H_1$ and $H_2$
\begin{equation}\label{eq:times for H1}
	\mathcal{O}\Big( (\alpha_1 \nu_{\max} + \alpha_2) T  + \log(1/\delta)\Big)
	= \mathcal{O}(\alpha_H \nu_{\max} T + \log(1/\delta))
\end{equation}
times (see \cite[Corollary 16]{ACL2023LCH2}), where $\alpha_H \ge \alpha_i$, $i=1,2$.

Applying the block-encoding circuit to the initial input state $\ket{0}_{a'}\ket{\tilde{\bm{W}}_0}$ gives
\[\text{SEL}_0\ket{0}_{a'}\ket{\tilde{\bm{W}}_0} = \ket{0}_{a'}\mathcal{U}^a(T)\ket{\tilde{\bm{W}}_0} + \ket{\bot},\]
where $\mathcal{U}^a(T)$ is the approximation of $\mathcal{U}(T)$ and $\tilde{\bm{W}}_0 = (\Phi ^{-1}\otimes I) \bm{W}_h(0)$. This step only needs one query to the state preparation oracle $O_{\tilde{w}}$ for $\tilde{\bm{W}}_0$.

According to the preceding discussions, we may conclude that there exists a unitary $V_0$ such that
\[\ket{0^{n_a}} \ket{0^w} \quad \xrightarrow{ V_0 } \quad  \frac{1}{\eta_0} \ket{0^{n_a}} \otimes \tilde{\bb{W}}_h^{a} + \ket{\bot},\]
where $\tilde{\bb{W}}_h^{a}$ is the approximate solution of $\tilde{\bm{W}}_h = (\Phi^{-1}\otimes I) \bm{W}_h$, given by
\begin{equation}\label{V0circuit}
	\tilde{\bb{W}}_h^{a} = \mathcal{U}^a(T)\tilde{\bb{W}}_0 \quad \mbox{and} \quad \eta_0 = \|\tilde{\bb{W}}_0\| = \|\bb{W}_h(0)\|.
	\end{equation}

%
\subsection{Complexity Analysis}

We can obtain $\ket{\bb{W}_h^{a}(T)}$ by measuring the state in \eqref{V0circuit} and obtaining all 0 in the first $n_a$ qubits, with the probability given by
\[\text{P}_{\text{r}}^0 = \Big( \frac{\|\tilde{\bb{W}}_h^{a}(T)\|}{\eta_0}  \Big)^2 \approx   \frac{\|\bb{W}_h(T)\|^2}{\|\bb{W}_h(0)\|^2}
= \frac{\|\bb{W}_h(T)\|^2}{C_e^2 \|T\bm{b}\|^2}, \quad C_e = \Big(\sum_{k=0}^{N_p-1} (\psi(p_k))^2  \Big)^{1/2}.\]
From Theorem \ref{thm:recovery u}, one can recover the target variables for $\bm{V}_f$ by performing a measurement in the computational basis:
\[M_k = \ket{k}\bra{k} \otimes I, \quad k \in \{j: p_j\geq p^{\Diamond}, \quad \e^{p_j} = \mathcal{O}(1)\}=:I_\Diamond,\]
where $I_\Diamond$ is referred to as the recovery index set. The state vector is then collapsed to
\[ \ket{\bm{W}_*} \equiv \ket{k_*} \otimes \frac{1}{\mathcal{N}}\Big(\sum\limits_i W_{k_*i} \ket{i} \Big) , \quad
\mathcal{N} = \Big(\sum\limits_i |W_{k_* i}|^2 \Big)^{1/2}\]
for some $k_*$ in the recovery index set $I_\Diamond$ with the probability
\begin{align*}
	\frac{\sum_i |W_{k_*i}(T)|^2}{\sum_{k,i} |W_{ki}(T)|^2}
	\approx \frac{\|\bm{W}(T,p_{k_*})\|^2}{\sum_k \|\bm{W}(T,p_k)\|^2}.
\end{align*}
Then the likelihood of acquiring $\ket{\bm{W}_*}$ that satisfies $k_* \in I_\Diamond$ is given by
\begin{equation*}
	\text{P}_{\text{r}}^*
	\approx \frac{ \sum_{k\in I_\Diamond} \|\bm{W}_h(T,p_k)\|^2}{\sum_k \|\bm{W}_h(T,p_k)\|^2}
	\approx \frac{C_{e0}^2\|\bb{V}_f(T)\|^2}{\| \bm{W}_h(T) \|^2}, \quad C_{e0} = \Big(\sum_{k\in I_{\Diamond}}  (\psi(p_k))^2  \Big)^{1/2}.
\end{equation*}
Since $\bm{V}_f(T) = \ket{0}^{\otimes 2}\otimes \bm{v}(T) +\ket{0}\ket{1}\otimes \bb{w}(T) +\ket{1}^{\otimes 2}\otimes (T\bm{b})$, one can perform a projection to get $\ket{\bm{v}(T)}$ with the probability $\frac{\|\bm{v}(T)\|^2}{\|\bm{V}_f(T)\|^2}$.
The overall probability for getting $\ket{\bm{v}(T)}$ is then approximated by
\[\text{P}_v = \frac{\|\bb{W}_h(T)\|^2}{C_e^2 \|T \bm{b}\|^2}\cdot \frac{C_{e0}^2\|\bb{V}_f(T)\|^2}{\| \bm{W}_h(T) \|^2} \cdot \frac{\|\bm{v}(T)\|^2}{\|\bm{V}_f(T)\|^2}
= \frac{C_{e0}^2}{C_e^2} \frac{\|\bm{v}(T)\|^2}{\|T\bm{b}\|^2}.\]
If $N_p$ is sufficiently large and $\psi(p) = e^{-|p|}$, we have
\[\triangle p C_{e0}^2 \approx  \int_{p^\Diamond}^{\infty} e^{-2p}   \d p
= \frac12 e^{-2p^\Diamond}, \qquad
\triangle p C_e^2 \approx  \int_{-\infty} ^{\infty}  e^{-2p} \d p= 1,\]
yielding
\begin{equation}\label{cece0}
	\frac{C_{e0}^2}{C_e^2} \lesssim  \frac12 e^{-2p^\Diamond} = \frac{1}{2e}.
\end{equation}
For other smooth initial data $\psi(p)$, the result is similar.

By using the amplitude amplification, the repeated times for the measurements can be approximated as
\begin{equation}\label{gtimes0}
	g = \mathcal{O}\Big(\frac{C_e}{C_{e0}} \frac{ T \|\bm{b}\|} {\|\bm{v}(T)\|} \Big) = \mathcal{O}\Big(\frac{C_e}{C_{e0}} \frac{T \|A \bm{x}\|} {\|\bm{x}\|}\Big) \lesssim 
   \kappa(A) \log \frac{1}{\varepsilon}.
\end{equation}

\begin{theorem}\label{thm:complexity}
	Suppose that
	$\psi(p)$ is smooth enough and $\triangle p =\mathcal{O}(\log \varepsilon)$, and the evolution time $T$ satisfies
	\[T = \Theta\Big( \frac{\log \frac{6}{\varepsilon}}{\sigma_{\min}(A)}\Big), \]
	with $\varepsilon$ being the desired precision.
	Then there exists a quantum algorithm that prepares an $\varepsilon$-approximation of the state $\ket{\bb{x}}$ with $\Omega(1)$ success probability and a flag indicating success, using
	\[ \mathcal{O}\Big(  \kappa^2 \log^2 \frac{1}{\varepsilon} \Big)\]
	queries to the block-encoding oracles for $H_1$ and $H_2$ and
	\[ \mathcal{O}\Big( \kappa \log \frac{1}{\varepsilon}\Big) \]
	queries to the state preparation oracles $O_w$ for $\bb{W}_h(0)$,
	where $\kappa = \kappa(A)$.
\end{theorem}
\begin{proof}
	Based on the preceding discussions, the state $\frac{1}{\eta_0 }\ket{0^{n_a}} \otimes  \bb{W}_h^a(T) + \ket{\bot}$ can be prepared, yielding the approximate state vector $\ket{\bm{v}^a(T)}$ for the solution $\bm{v}(T)$.
	The error between $\ket{\bm{x}}$ and $\ket{\bb{v}^a(T)}$ can be split as
	\begin{equation*}
		\begin{aligned}
			\| \ket{\bb{x}} - \ket{\bb{v}^a(T)} \| &\le \| \ket{\bb{x}} - \ket{\bb{v}(T)} \| +
			\| \ket{\bb{v}(T)} - \ket{\bb{v}_h(T)}\|+
			\| \ket{\bb{v}_h(T)} - \ket{\bb{v}^a(T)}\| \\
			&=: E_1 + E_2 +E_3.
		\end{aligned}
	\end{equation*}
	Here $\bb{v}_h = e^{p_k}\cdot (\bra{0}^{\otimes 2}\cdot I_N) (\bra{0}\otimes I)\bb{W}_h(T)$.
	For $E_1$, using Theorem \ref{thm:ODESS} and the inequality $\| \bb{x}/{\|\bb{x}\|} - \bb{y}/{\|\bb{y}\|} \| \le 2 {\|\bb{x} - \bb{y}\|}/{\|\bb{x}\|}$ for two vectors $\bb{x}, \bb{y}$, we have
	\[E_1 = \| \ket{\bb{x}} - \ket{\bb{v}(T)} \|\le  2 \| \bb{x}  - \bb{v}(T) \|/\| \bb{x}  \| \le \varepsilon/3 .\]
	For $E_2$, the error is from the spatial discretization. 
	According to \cite[Theorem 4.4]{Jin2024schrodingerization}, the error between $\bm{v}_h^d$ and $\bm{v}$ satisfies
	\[E_2 = \| \ket{\bb{v}(T)} - \ket{\bb{v}_h(T)} \|\le  2 \| \bb{v}(T)  - \bb{v}_h(T) \|/\| \bb{v}(T)  \| \leq  \varepsilon/3.\]
	For $ E_3 $, neglecting the error in the block-encoding, we obtain
	\begin{align*}
		E_3
		= \| \ket{\bb{v}_h(T)} - \ket{\bb{v}^a(T)}\| \le  2 \| \bb{v}_h(T) - \bb{v}^a(T) \|/\| \bb{v}_h(T) \|.
	\end{align*}
	To ensure that the total error, i.e., the sum of $E_1$, $E_2$ and $E_3$, remains controlled within $\varepsilon$, we require $E_3\leq \varepsilon/3$, which leads to the following inequality:
	\begin{equation*}
		\begin{aligned}
		 \|\bm{W}_h(T) - \bm{W}^a(T)\| \lesssim e^{-p_k}\|\bb{v}_h(T) - \bb{v}^a(T)\| 
		   \lesssim \varepsilon \|\bm{v}_h\|
			:=\delta.
		\end{aligned}
	\end{equation*}
	Here  $p_k$ can be chosen as $\mathcal{O}(1)$ since the probability of projecting onto $ p_k $ near $ p^\Diamond = \frac12$ is comparable to that of projecting onto the entire range $ p_k \ge p^\Diamond $ due to the exponential decay of $ e^{-p} $.

	For simplicity, we assume negligible error from the quantum Fourier transform.
	According to the discussion in Section \ref{subsec:implementation}, there exists a quantum algorithm which maps $\ket{0^{n_a}} \otimes \ket{0^{w}}$ to the state $\frac{1}{\eta_0 }\ket{0^{n_a}} \otimes  \bb{W}_h^a(T) + \ket{\bot}$ such that $\bb{W}_h^a(T)$ is a $\delta$-approximation of $\bb{W}_h(T)$, using
	\begin{itemize}
		\item $\mathcal{O}\big( \frac{\alpha_H  \nu_{\max} }{\sigma_{\min}(A)} \log \frac{1}{\varepsilon} + \log \frac{\eta_0}{ \delta}\big)$ queries to the block-encoding oracles for $H_1$ and $H_2$,
		\item $\mathcal{O}(1)$ queries to the state preparation oracle $O_w$ for $\bb{W}_h(0)$,
	\end{itemize}
	from \eqref{eq:times for H1}. In addition,
	we have
	\[\frac{\eta_0}{\delta} \lesssim
	\frac{\|\bm{W}_h(0)\|}{\varepsilon\|\bb{v}_h(T)\|}
	\approx \frac{C_eT\|\bb{b}\|}{\varepsilon \|\bb{x}\|}
	\leq \mathcal{O}(\frac{\kappa}{ \varepsilon} \log^2 \frac{1}{\varepsilon}),\]
	with $\triangle p = \mathcal{O}(\log \varepsilon)$, the proof is finished by multiplying the repeated times in \eqref{gtimes0}.
\end{proof}

On the other hand, it's well known that  $h\sim N^{-1/d}$.
And we set $kh=\mathcal{O}(1)$ which leads to $N\sim k^d$ and $\kappa=\mathcal{O}(h^{-2})=\mathcal{O}(k^{2})$.
Then, it has the following table for time complexity comparison of classical and quantum methods.
\begin{table}[h]
\centering
\small
\caption{Time complexity comparison of methods for solving Helmholtz equation}
\begin{tabular}{lll}
\toprule
\textbf{Method} & \textbf{Complexity} & \textbf{Condition Number Dependence} \\
\midrule 
CG & $\mathcal{O}(k^{2+d}\log(\varepsilon^{-1}))$ & $\mathcal{O}(N\kappa)$ \\
Original HHL & $\tilde{\mathcal{O}}(k^4\varepsilon^{-1})$ & $\mathcal{O}(\kappa^2)$\\
Schr{\"o}dingerization & $\tilde{\mathcal{O}}(k^4\log^2(\varepsilon^{-1}))$ & $\mathcal{O}(\kappa^2)$\\  
\bottomrule
\end{tabular}
\label{table1}
\end{table} 

From the Table \ref{table1}, it can be observed that Schr{\"o}dingerization outperforms the HHL algorithm and is superior to the CG method in high-dimensional cases ($d>2$). 
Additionally, while larger matrices drastically increase space complexity, quantum computers can greatly reduce it thanks to the parallelism of qubits.

\section{Quantum preconditioning method for the Hel- mholtz equations}\label{sec:Quantum_preconditioning_method}

Although in section \ref{sec:Schrforindefinite}, we are able to regularize indefinite systems with almost no additional cost, the high condition number will still affect the complexity on our methods, see Table \ref{table1}.
If an effective preconditioner $P$ can be inexpensively constructed, we can turn to solving \eqref{1stODE} by letting 
\begin{equation}\label{Quantum_precondition}
\bm{V}=\left[\begin{array}{c}
    \bm{v}\\
    \bm{w}
\end{array}\right],\ 
M=\left[\begin{array}{cc}
    O_N & -A^{\dagger}P^{\dagger}\\
    PA & -\gamma I_N
\end{array}\right],\ 
F=\left[\begin{array}{c}
    \bm{0}\\
    -P\bm{b}
\end{array}\right],
\end{equation}
where $\gamma=2\sigma_{\min}(PA)$.
Naturally, the ideal scenario would be to possess $\kappa(PA)=\mathcal{O}(1)$.

In classical computing, the shifted Laplace preconditioners (SLPs) are popular preconditioners for the Helmholtz equation and can significantly improve the iterative solution of the Helmholtz equation by reducing its condition number.
Their development started with the preconditioner obtained by discretizing Laplace operator $(C=-\Delta_h)$. 
Subsequently, a Laplace operator with a complex shift was introduced in \cite{Erlangga2006Multigrid} and found to be more effective.
Now, the complex shifted Laplace preconditioner (CSLP) becomes a fundamental tool for solving indefinite Helmholtz equations that arise in wave propagation and scattering problems. 
It modifies the classical Laplace operator through a complex-valued shift
$$C=-\Delta_h-(k^2+i\alpha)I$$
where $\alpha\in\mathbb{R}$ and $\alpha$ is an empirically chosen damping parameter. 
Adding an imaginary shift moves eigenvalues away from zero, improving the condition number of the preconditioned system.
Moreover, the complex shift introduces damping and renders the preconditioned system amenable to approximate inversion using either geometric multigrid~\cite{Yogi2008} or ILU~\cite{Magolu2001}. 
More recently, algebraic multigrid has been used to invert the preconditioner~\cite{Tuomas2007,Matthias2009}.
If the shift parameter $\alpha$ is too small, it will lead to inadequate preconditioning effect.
On the other hand, over-damping will distort the physical solution from too large shift.
Thus, the optimal parameter has attracted a lot of research interests~\cite{Erlangga2004Helmholtz,Erlangga2006Comparison,Erlangga2006Multigrid,Gijzen2007,Reps2010,Sheikh2013,Gander2015,Cocquet2017}.

However, implementing the inverse of a complex-shifted Laplacian in quantum computing is highly expensive. 
This is because classical techniques like multigrid methods-commonly used in classical computing-are not easily adaptable to quantum computers. 
Moreover, the oscillations are global in nature, making it impossible to approximate its inverse well with any sparse matrix.
Also, it has $\kappa(C)=\order{k^2}$ which leads to high computational cost.
However, our research reveals that employing preconditioner
\begin{equation}\label{preconditioner}
P^{-1}=-\Delta_h+k^2I_h\quad \textit{or}\quad P^{-1}=-\Delta_h+ik^2I_h,
\end{equation}
while not achieving the optimal $\kappa(PA)$, can reduce it to $\kappa(PA)=\mathcal{O}(k)$.
To demonstrate this, one can use a simple but rough analysis as follows.
First of all, we find the eigenvalues of $PA$ to be
\begin{equation}\label{eigsofPA}
\lambda_j(PA)=\frac{\mu_j^2-k^2}{\mu_j^2 +k^2}\quad\textit{or}\quad\frac{\mu_j^2-k^2}{\mu_j^2 +ik^2},
\end{equation}
where $\mu_j^2$ is the eigenvalue of negative discrete Laplacian operator $-\Delta_h$.
It is easy to check that $\vert\lambda_{j}(PA)\vert\le 1$ and $\vert\lambda_{j}(PA)\vert\to 1$ when $h\to 0$.
To estimate the minimum eigenvalue, it is assumed
that one of the eigenvalues, $\mu_s$ is close (but not equal) to $k$ (consistent with the hypothesis in Section \ref{discretization}). 
To be more precise, let $\mu_s=k+\delta_0$, where $\delta_0$ is a small number which leads to $\delta_0 \ll k$.  
If this relation is substituted into \eqref{eigsofPA}, and the higher order terms are neglected,  then we can find
$$\vert\lambda_{s}(PA)\vert=\frac{\vert\delta_0\vert}{k}\quad\textit{or}\quad\frac{\sqrt{2}\vert\delta_0\vert}{k}.$$

 Moreover, since all eigenvalues of this matrix have positive real parts and a lower condition number ($\kappa(P)=\mathcal{O}((kh)^{-2})$), its preparation cost remains relatively low.

\begin{corollary}\label{precondition_coro}
Suppose the assumptions in Theorem~\ref{thm:complexity} hold.
By using preconditioner \eqref{preconditioner}, the complexity of the quantum preconditioned system \eqref{Quantum_precondition} is
$$\tilde{\mathcal{O}}(k^2\log^2(\varepsilon^{-1})).$$
\end{corollary}

A practical approach to solve $P$ is the QSVT~\cite{Gilyen2019QSVT}. 
Given that $P$ has a modest condition number here, the computational complexity is consequently low.

\section{Numerical experiments}\label{sec:numerical tests}

In this section, we present some experiments for Helmholtz equation.
We consider the one-dimensional Helmholtz equation on $\Omega=[0,1]$ with the Robin boundary condition (derived from the Sommerfeld radiation boundary
condition) at $r=1$:
\begin{equation*}
\begin{aligned}
&-\Delta{u}-k^2u=f,\ \textit{in}\ (0,1)\\
&u(0)=0,\ u'(1)-iku(1)=0.
\end{aligned}
\end{equation*}
Then we assume the source term $f$ to be
$$f=-\sin(kx),$$
which gives the closed form solution
$$u=-\frac{x\cos(kx)}{2k}+\sin(kx)\left(\frac{1+2e^{2ik}-2ik}{4k^2}\right).$$
By employing discretization scheme \eqref{1dcorrectionFDM}, no dispersion error leads to no pollution effect as well.
The matrix $A$ is \eqref{matA} with Robin boundary
condition and step size $h=2^{-n}$ satisfying $kh<1$.
For the discrete Fourier transform in Schr\"odingerization method, we take $R=-L=5\pi$ and $M=2^m$.
Consequently, a larger step size $h$ can reduce the termination time $T$, thereby decreasing the overall computational complexity. 
In practical applications, the value of the product $kh$ must be carefully selected to balance accuracy and efficiency.

\begin{figure}[!ht]
\centering
\subfigure[]{
\includegraphics[scale=.5]{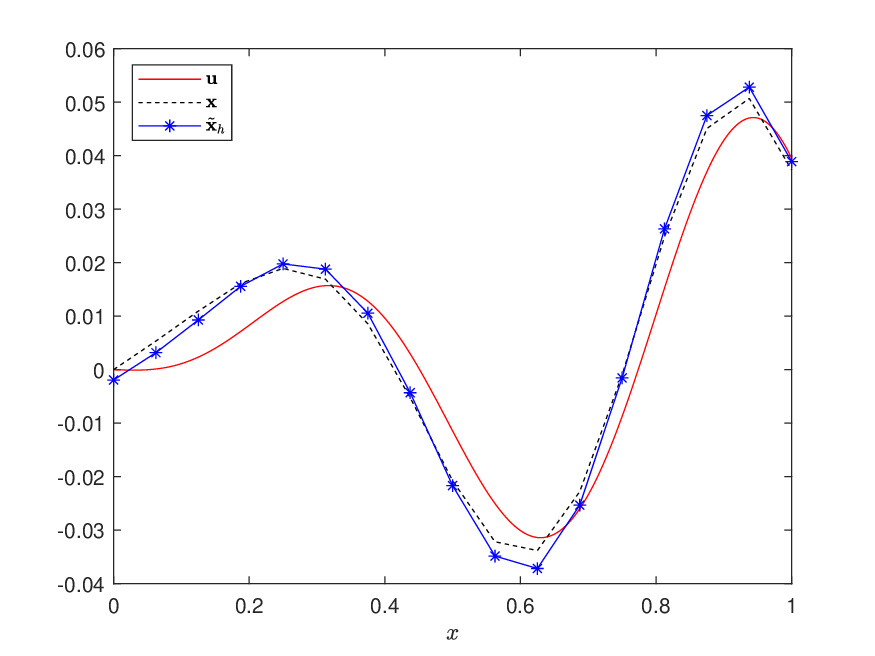}
}
\subfigure[]{
\includegraphics[scale=.5]{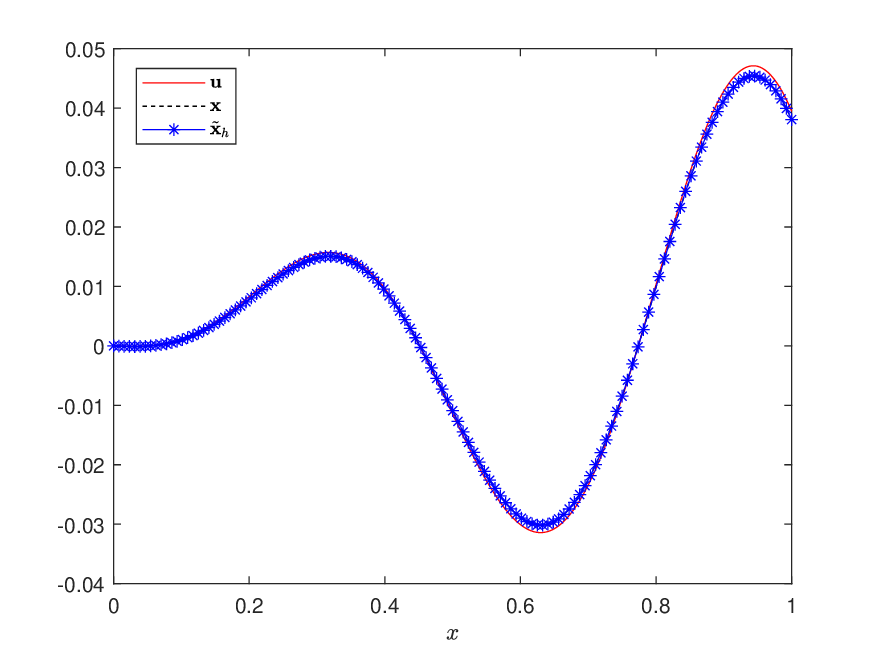}
}
\caption{Comparison between a reference solution and the numerical solution of Schr{\"o}dingerization for $k=10,m=8$. (a): $T=1\times 10^3,n=4$ ($kh=0.625$); (b): $T=5\times 10^3,n=7$ ($kh=0.078125$).}
\label{imgeg21}
\end{figure}

\begin{figure}[!ht]
\centering
\subfigure[]{
\includegraphics[scale=.5]{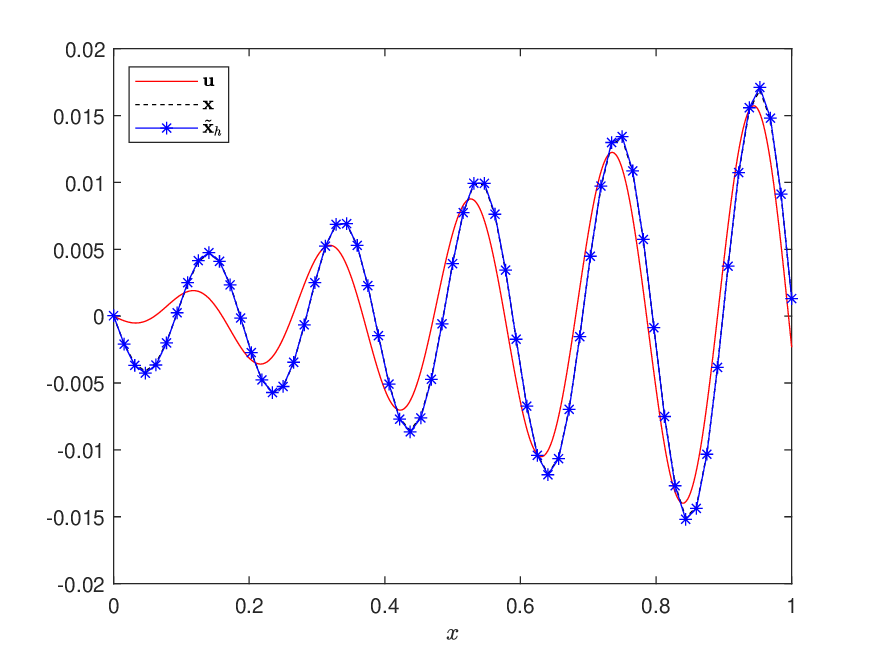}
}
\subfigure[]{
\includegraphics[scale=.5]{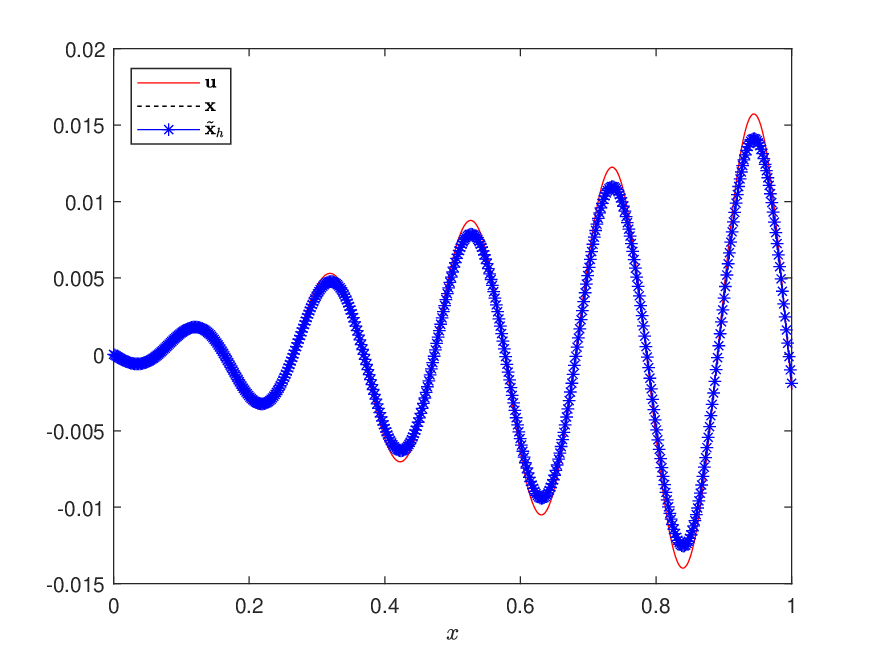}
}
\caption{Comparison between a reference solution and the numerical solution of Schr{\"o}dingerization for $k=30,m=8$. (a): $T=5\times 10^3,n=6$ ($kh=0.46875$); (b): $T=2\times 10^4,n=9$ ($kh=0.05859375$).}
\label{imgeg22}
\end{figure}

\begin{figure}[!ht]
\centering
\subfigure[]{
\includegraphics[scale=.35]{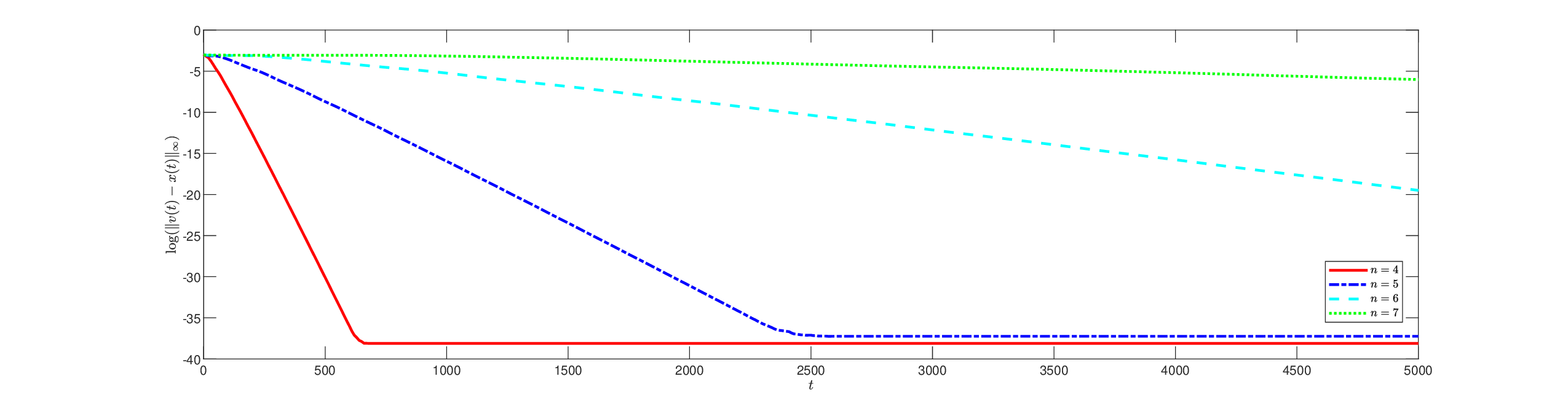}}
\subfigure[]{
\includegraphics[scale=.35]{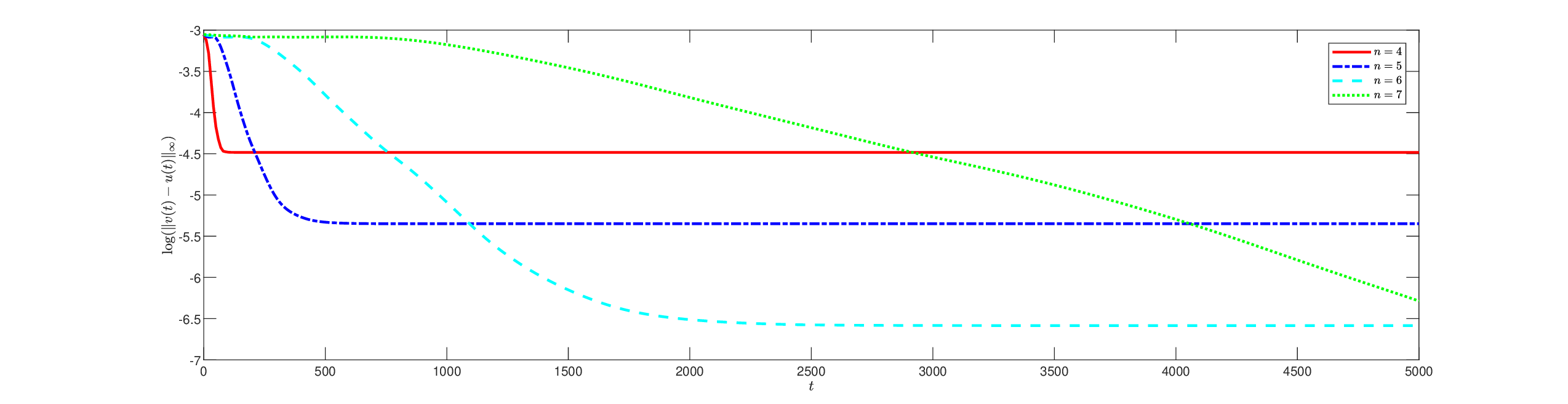}}
\caption{The convergence rates for $n=4,5,6,7$ when $k=10$. (a): $\log(\Vert\bm{v}(t)-\bm{x}(t)\Vert_\infty)$. (b): $\log(\Vert\bm{v}(t)-u(t)\Vert_\infty)$.}
\label{imgeg23}
\end{figure}

Let $\bm{x}_h$ denote the numerical solution obtained through the recovery procedure in Theorem \ref{thm:recovery u}.
From Figures \ref{imgeg21} and \ref{imgeg22}, it is evident that the oscillation of the solution can be captured by our Schr{\"o}dingerization-based approach, and the accuracy increases as $h$ decreases.
On the other hand, a finer step size requires a corresponding increase in the final time $T$.
In Figure \ref{imgeg23}, it is observed that while a smaller step size $h$ enhances the accuracy of $\bm{v}(t)$ in \eqref{1stODE} toward converging to the exact solution $u$, it simultaneously slows the convergence rate to the numerical solution $\bm{x}$. 
For the case $k=10$, the choice $n=6$ provides the optimal balance between these competing effects.
It confirms our theoretical results from Section \ref{sec:Schrforindefinite}.

\begin{figure}[!ht]
\centering
\subfigure[]{
\includegraphics[scale=.5]{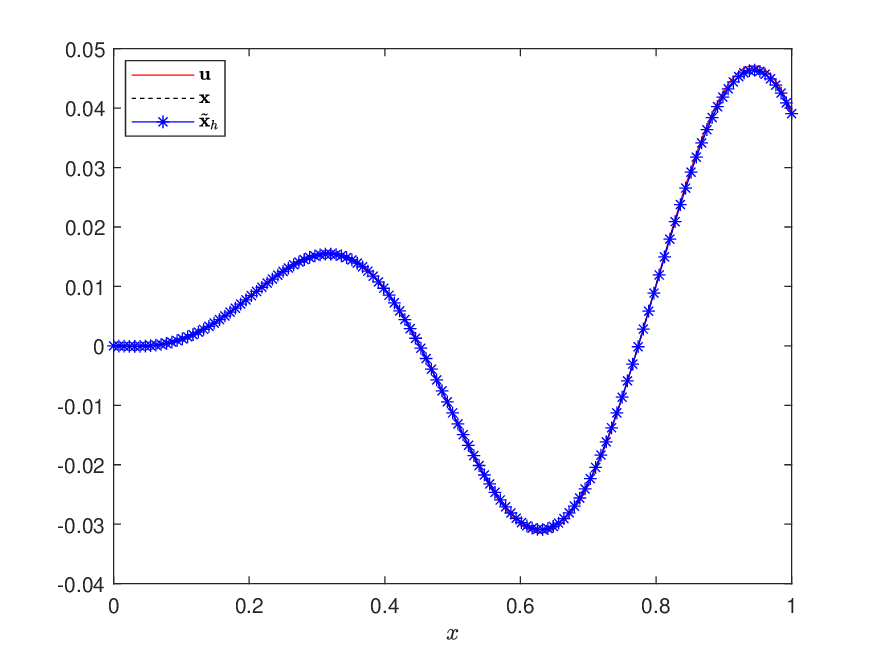}
}
\subfigure[]{
\includegraphics[scale=.5]{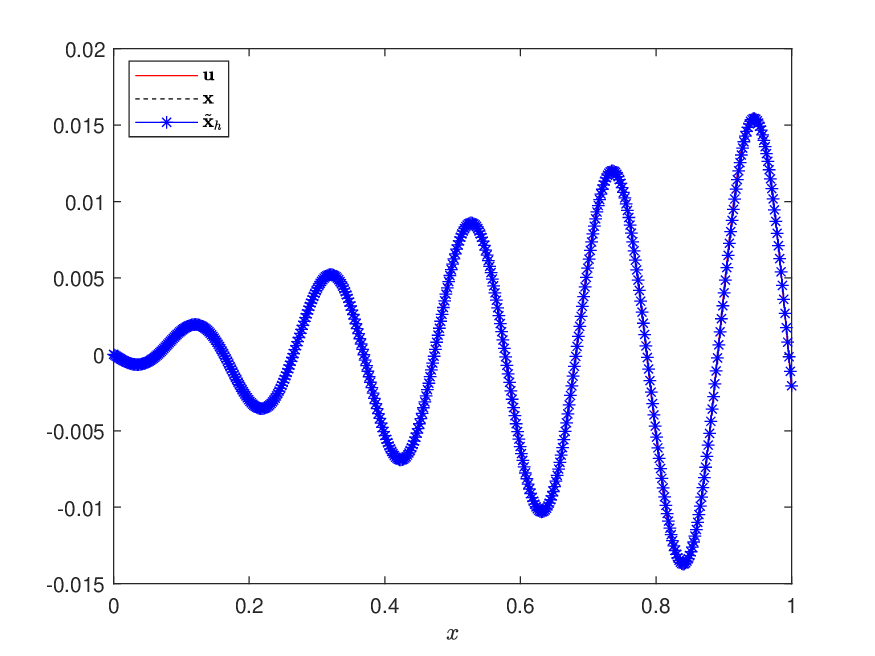}
}
\caption{Comparison between a reference solution and the numerical solution of Schr{\"o}dingerization (with preconditioner \eqref{preconditioner}) for $m=8$. (a): $k=10,T=50,n=6$; (b): $k=30,T=150,n=9$.}
\label{imgeg24}
\end{figure}

Then we adopt the preconditioning technique in section~\ref{sec:Quantum_preconditioning_method},   see Figure \ref{imgeg24}.
Due to the reduction in condition number, the termination time can be correspondingly reduced to $\mathcal{O}(k)$. 
This proves particularly advantageous for mid- and high-frequency problems.

\section{Conclusion}

The classical numerical methods for the Helmholtz equations are subject to the larger matrices (which leads to huge space and time complexity) stemming from
condition numbers and dimensions, for high wave number problems.
However, quantum computing can address these issues.
In this paper, we present quantum algorithms to solve Helmholtz equations based on the classical DDS method \eqref{steadystateeq2}, utilizing
Schr{\"o}dingerization. 
The implementation details and complexity analysis are provided. 
The total query complexity is $\mathcal{O}(\kappa^2\textit{polylog}\frac{1}{\varepsilon})$, where $\kappa=\kappa(A)$ and $\varepsilon$ represents the desired accuracy. 
Moreover, our framework is suitable for all indefinite systems and leads to a dependence on the condition number reducing from $\kappa(B)=\kappa^2(A)$ to $\sqrt{\kappa(B)}=\kappa(A)$ for the
Schr{\"o}dingerization-based approach. 
We have shown that Schr{\"o}dingerization outperforms the classic iterative algorithms in high-dimensional cases. 
By further exploiting the properties of preconditioned matrices, it will lead to a better path to solve the high-frequency problems.

Admittedly, the preconditioner currently adopted represents only an initial result and remains suboptimal. Consequently, constructing efficient quantum algorithms for advanced preconditioners (e.g., CSLP) will constitute our primary focus in the subsequent research phase.

\section*{Acknowledgement}

SJ was supported by NSFC grant Nos. 12341104 and 12426637, the Shanghai Science and Technology Innovation Action Plan 24LZ1401200,  
the Shanghai Jiao Tong University 2030 Initiative, and the Fundamental Research Funds for the Central Universities. 


\bibliographystyle{plain}
\bibliography{refs}

\begin{thebibliography}{10}

\bibitem{Tuomas2007}
T.~Airaksinen, E.~Heikkola, A.~Pennanen, and J.~Toivanen.
\newblock An algebraic multigrid based shifted-laplacian preconditioner for the
  helmholtz equation.
\newblock {\em J. Comput. Phys.}, 226(1):1196--1210, 2007.

\bibitem{ACL2023LCH2}
D.~An, A.~W. Childs, and L.~Lin.
\newblock Quantum algorithm for linear non-unitary dynamics with near-optimal
  dependence on all parameters.
\newblock {\em Arxiv:2505.00370v1}, 2025.

\bibitem{Berry2014Highorder}
D.~W. Berry.
\newblock High-order quantum algorithm for solving linear differential
  equations.
\newblock {\em Journal of Physics A: Mathematical and Theoretical},
  47(10):105301, feb 2014.

\bibitem{Matthias2009}
M.~Bollh\"{o}fer, M.~J. Grote, and O.~Schenk.
\newblock Algebraic multilevel preconditioner for the helmholtz equation in
  heterogeneous media.
\newblock {\em SIAM J. Sci. Comput.}, 31(5):3781--3805, 2009.

\bibitem{Childs2017LCU}
A.~M. Childs, R.~Kothari, and R.~D. Somma.
\newblock Quantum algorithm for systems of linear equations with exponentially
  improved dependence on precision.
\newblock {\em SIAM J. Comput.}, 46(6):1920--1950, 2017.

\bibitem{Childs2020Spectral}
A.~M. Childs and J.~P. Liu.
\newblock Quantum spectral methods for differential equations.
\newblock {\em Comm. Math. Phys. C}, 375:1427--1457, 2020.

\bibitem{Childs2021highprecision}
A.~M. Childs, J.~P. Liu, and A.~Ostrander.
\newblock High-precision quantum algorithms for partial differential equations.
\newblock {\em {Quantum}}, 5:574, November 2021.

\bibitem{Cocquet2017}
P.~H. Cocquet and M.~J. Gander.
\newblock How large a shift is needed in the shifted helmholtz preconditioner
  for its effective inversion by multigrid?
\newblock {\em SIAM J. Sci. Comput.}, 39(2):A438--A478, 2017.

\bibitem{Cocquet2024Asymptotic}
P.~H. Cocquet and M.~J. Gander.
\newblock Asymptotic dispersion correction in general finite difference schemes
  for helmholtz problems.
\newblock {\em SIAM J. Sci. Comput.}, 46(2):A670--A696, 2024.

\bibitem{Costa2019wave}
P.~C.~S. Costa, S.~Jordan, and A.~Ostrander.
\newblock Quantum algorithm for simulating the wave equation.
\newblock {\em Phys. Rev. A}, 99:012323, Jan 2019.

\bibitem{Deutsch1985quantum}
D.~Deutsch.
\newblock Quantum theory, the church--turing principle and the universal
  quantum computer.
\newblock {\em Proceedings of the Royal Society of London. A. Mathematical and
  Physical Sciences}, 400(1818):97--117, 1985.

\bibitem{DiVincenzo1995quantum}
D.~P DiVincenzo.
\newblock Quantum computation.
\newblock {\em Science}, 270(5234):255--261, 1995.

\bibitem{Ekert1998quantum}
A.~Ekert and R.~Jozsa.
\newblock Quantum algorithms: entanglement--enhanced information processing.
\newblock {\em Philosophical Transactions of the Royal Society of London.
  Series A: Mathematical, Physical and Engineering Sciences},
  356(1743):1769--1782, 1998.

\bibitem{Engel2019Vlasov}
A.~Engel, G.~Smith, and S.~E. Parker.
\newblock Quantum algorithm for the vlasov equation.
\newblock {\em Phys. Rev. A}, 100:062315, Dec 2019.

\bibitem{Yogi2008}
Y.~A. Erlangga and R.~Nabben.
\newblock On a multilevel krylov method for the helmholtz equation
  preconditioned by shifted laplacian.
\newblock {\em Electron. Trans. Numer. Anal.}, 31:403--424, 2008.

\bibitem{Erlangga2006Multigrid}
Y.~A. Erlangga, C.~W. Oosterlee, and C.~Vuik.
\newblock A novel multigrid based preconditioner for heterogeneous helmholtz
  problems.
\newblock {\em SIAM J. Sci. Comput.}, 27(4):1471--1492, 2006.

\bibitem{Erlangga2004Helmholtz}
Y.~A. Erlangga, C.~Vuik, and C.~W. Oosterlee.
\newblock On a class of preconditioners for solving the helmholtz equation.
\newblock {\em Appl. Numer. Math.}, 50(3):409--425, 2004.

\bibitem{Erlangga2006Comparison}
Y.~A. Erlangga, C.~Vuik, and C.~W. Oosterlee.
\newblock Comparison of multigrid and incomplete lu shifted-laplace
  preconditioners for the inhomogeneous helmholtz equation.
\newblock {\em Appl. Numer. Math.}, 56(5):648--666, 2006.

\bibitem{Feynman1982}
R.~P. Feynman.
\newblock Simulating physics with computers.
\newblock {\em Internat. J. Theoret. Phys.}, 21:467--488, 1982.

\bibitem{Freiser1969}
M.~Freiser and P.~Marcus.
\newblock A survey of some physical limitations on computer elements.
\newblock {\em IEEE Transactions on Magnetics}, 5(2):82--90, 1969.

\bibitem{Gander2015}
M.~J. Gander, I.~G. Graham, and E.~A. Spence.
\newblock Applying gmres to the helmholtz equation with shifted laplacian
  preconditioning: what is the largest shift for which wavenumber-independent
  convergence is guaranteed?
\newblock {\em Numer. Math.}, 131:567--614, 2015.

\bibitem{Gilyen2019QSVT}
A.~Gily\'{e}n, Y.~Su, G.~H. Low, and N.~Wiebe.
\newblock Quantum singular value transformation and beyond: exponential
  improvements for quantum matrix arithmetics.
\newblock In {\em Proceedings of the 51st Annual ACM SIGACT Symposium on Theory
  of Computing}, STOC 2019, page 193–204, New York, NY, USA, 2019.
  Association for Computing Machinery.

\bibitem{gilyen2019quantum}
A.~Gily{\'e}n, Y.~Su, G.~H. Low, and N.~Wiebe.
\newblock Quantum singular value transformation and beyond: exponential
  improvements for quantum matrix arithmetics.
\newblock {\em In {\em Proceedings of the 51st Annual ACM SIGACT Symposium on
  Theory of Computing}}, pages 193--204, 2019.

\bibitem{Harrow2009Quantum}
A.~W. Harrow, A.~Hassidim, and S.~Lloyd.
\newblock Quantum algorithm for linear systems of equations.
\newblock {\em Phys. Rev. Lett.}, 103:150502, Oct 2009.

\bibitem{Hu2024Multiscale}
J.~Hu, S.~Jin, and L.~Zhang.
\newblock Quantum algorithms for multiscale partial differential equations.
\newblock {\em Multiscale Model. Simul.}, 22(3):1030--1067, 2024.

\bibitem{Ihlenburg1995FEM}
F.~Ihlenburg and I.~Babuška.
\newblock Finite element solution of the helmholtz equation with high wave
  number part i: The h-version of the fem.
\newblock {\em Comput. Math. Appl.}, 30(9):9--37, 1995.

\bibitem{Jin2024schrodingerization}
S.~Jin, N.~Liu, and C.~Ma.
\newblock On schr{\"o}dingerization based quantum algorithms for linear
  dynamical systems with inhomogeneous terms.
\newblock {\em arXiv:2402.14696}, 2024.

\bibitem{Jin2024illposed}
S.~Jin, N.~Liu, and C.~Ma.
\newblock Schr\"odingerisation based computationally stable algorithms for
  ill-posed problems in partial differential equations.
\newblock {\em arXiv:2403.19123}, 2024.

\bibitem{JLMY252}
S.~Jin, N.~Liu, C.~Ma, and Y.~Yu.
\newblock On the schr\"odingerization method for linear non-unitary dynamics
  with optimal dependence on matrix queries.
\newblock {\em Arxiv:2505.00370v1}, 2025.

\bibitem{JLMY25}
S.~Jin, N.~Liu, C.~Ma, and Y.~Yu.
\newblock Quantum preconditioning method for linear systems problems via
  schr\"odingerization.
\newblock {\em Arxiv:2505.06866}, 2025.

\bibitem{Jin2022quantum}
S.~Jin, N.~Liu, and Y.~Yu.
\newblock Quantum simulation of partial differential equations via
  schr{\"o}dingerisation, 2022.

\bibitem{Jin2023Quantum}
S.~Jin, N.~Liu, and Y.~Yu.
\newblock Quantum simulation of partial differential equations: {A}pplications
  and detailed analysis.
\newblock {\em Phys. Rev. A}, 108:032603, Sep 2023.

\bibitem{Jin2024PRL}
S.~Jin, N.~Liu, and Y.~Yu.
\newblock Quantum simulation of partial differential equations via
  schr\"odingerization.
\newblock {\em Phys. Rev. Lett.}, 133:230602, Dec 2024.

\bibitem{Klappenecker2001DST}
A.~Klappenecker and M.~Rotteler.
\newblock Discrete cosine transforms on quantum computers.
\newblock In {\em ISPA 2001. Proceedings of the 2nd International Symposium on
  Image and Signal Processing and Analysis. In conjunction with 23rd
  International Conference on Information Technology Interfaces (IEEE Cat.},
  pages 464--468, 2001.

\bibitem{Low2019QSP}
G.~H. Low and I.~L. Chuang.
\newblock Hamiltonian {S}imulation by {Q}ubitization.
\newblock {\em {Quantum}}, 3:163, July 2019.

\bibitem{Magolu2001}
M.~M.~M. Made.
\newblock Incomplete factorization-based preconditionings for solving the
  helmholtz equation.
\newblock {\em Internat. J. Numer. Methods Engrg.}, 50(5):1077--1101, 2001.

\bibitem{Melenk2011Galerkin}
J.~M. Melenk and S.~Sauter.
\newblock Wavenumber explicit convergence analysis for galerkin discretizations
  of the helmholtz equation.
\newblock {\em SIAM J. Numer. Anal.}, 49(3):1210--1243, 2011.

\bibitem{Montanaro2016FEM}
A.~Montanaro and S.~Pallister.
\newblock Quantum algorithms and the finite element method.
\newblock {\em Phys. Rev. A}, 93:032324, Mar 2016.

\bibitem{Reps2010}
B.~Reps, W.~Vanroose, and H.~Zubair.
\newblock On the indefinite helmholtz equation: Complex stretched absorbing
  boundary layers, iterative analysis, and preconditioning.
\newblock {\em J. Comput. Phys.}, 229(22):8384--8405, 2010.

\bibitem{Sheikh2013}
A.~H. Sheikh, D.~Lahaye, and C.~Vuik.
\newblock On the convergence of shifted laplace preconditioner combined with
  multilevel deflation.
\newblock {\em Numer. Linear Algebra Appl.}, 20(4):645--662, 2013.

\bibitem{Shor1994}
P.~W. Shor.
\newblock Algorithms for quantum computation: discrete logarithms and
  factoring.
\newblock In {\em Proceedings 35th Annual Symposium on Foundations of Computer
  Science}, pages 124--134, 1994.

\bibitem{Spence2023FEM}
E.~A. Spence.
\newblock A simple proof that the hp-fem does not suffer from the pollution
  effect for the constant-coefficient full-space helmholtz equation.
\newblock {\em Adv. Comput. Math.}, 49(2):27--49, 2023.

\bibitem{Andrew1998}
A.~Steane.
\newblock Quantum computing.
\newblock {\em Rep. Progr. Phys.}, 61(2):117, feb 1998.

\bibitem{Gijzen2007}
M.~B. van Gijzen, Y.~A. Erlangga, and C.~Vuik.
\newblock Spectral analysis of the discrete helmholtz operator preconditioned
  with a shifted laplacian.
\newblock {\em SIAM J. Sci. Comput.}, 29(5):1942--1958, 2007.

\bibitem{Zhu2013FEM}
L.~Zhu and H.~Wu.
\newblock Preasymptotic error analysis of cip-fem and fem for helmholtz
  equation with high wave number. part ii: $hp$ version.
\newblock {\em SIAM J. Numer. Anal.}, 51(3):1828--1852, 2013.

\end{thebibliography}
\end{document}